\documentclass[12pt]{amsart}
\usepackage{euscript,epsf,amssymb,xr}

\let\oldmarginpar\marginpar
\renewcommand\marginpar[1]
{\oldmarginpar{\tiny\bf \begin{flushleft} #1 \end{flushleft}}}

\input xy
\xyoption{all}

\newcommand{\la}{\langle}
\newcommand{\ra}{\rangle}

\newtheorem{theorem}{\bf Theorem}[section]
\newtheorem{lemma}[theorem]{\bf Lemma}

\newtheorem{prop}[theorem]{\bf Proposition}
\newtheorem{corollary}[theorem]{\bf Corollary}

\newcommand{\CC}{{\Bbb C}}

\newcommand{\NN}{{\Bbb N}}
\newcommand{\PP}{{\Bbb P}}
\newcommand{\QQ}{{\Bbb Q}}
\newcommand{\RR}{{\Bbb R}}

\newcommand{\ZZ}{{\Bbb Z}}

\newcommand{\ggreat}{>\kern-.7ex>}
\newcommand{\ssmall}{<\kern-.7ex<}
\newcommand{\qu}{/\kern-.7ex/}
\newcommand{\exh}{\to\kern-1.8ex\to}

\newcommand{\hH}{{\EuScript{H}}}

\newcommand{\GL}{\operatorname{GL}}

\newcommand{\Diff}{\operatorname{Diff}}

\newcommand{\Id}{\operatorname{Id}}
\newcommand{\Ind}{\operatorname{Ind}}

\newcommand{\Ker}{\operatorname{Ker}}

\newcommand{\rk}{\operatorname{rk}}
\newcommand{\SO}{\operatorname{SO}}
\newcommand{\SU}{\operatorname{SU}}

\newcommand{\ov}{\overline}
\newcommand{\un}{\underline}
\newcommand{\uC}{\underline{\CC}}

\newcommand{\ord}{\operatorname{ord}}

\newcommand{\wt}{\widetilde}
\newcommand{\imag}{{\mathbf i}}

\title[Non Jordan groups of diffeomorphisms and compact Lie groups]
{Non Jordan groups of diffeomorphisms and actions of compact Lie groups on manifolds}

\author{Ignasi Mundet i Riera}
\address{Departament d'\`Algebra i Geometria\\
Facultat de Matem\`atiques\\
Universitat de Barcelona\\
Gran Via de les Corts Catalanes 585\\
08007 Barcelona \\
Spain}
\email{ignasi.mundet@ub.edu}

\date{December 18, 2014}

\begin{document}

\maketitle

\begin{abstract}
A recent preprint of Csik\'os, Pyber and Szab\'o \cite{CPS}
proves that the diffeomorphism group of $T^2\times S^2$ is not
Jordan. The purpose of this paper is to
generalize the arguments of Csik\'os, Pyber and Szab\'o in
order to obtain many other examples of compact manifolds whose
diffeomorphism group fails to be Jordan. In particular we prove
that for any $\epsilon>0$ there exist manifolds admitting
effective actions of arbitrarily large $p$-groups $\Gamma$ all
of whose abelian subgroups have at most $|\Gamma|^{\epsilon}$
elements. Finally, we also recover some results on nonexistence
of effective actions of compact connected semisimple Lie group
on manifolds.
\end{abstract}

\section{Introduction}

A group $G$ is said to be Jordan if there is some constant $C$ such
that any finite subgroup $\Gamma$ of $G$ contains an abelian subgroup whose index
in $\Gamma$ is at most $C$, see \cite{Po0}.
In their preprint \cite{CPS},
Csik\'os, Pyber and Szab\'o prove that the diffeomorphism
group of the product of the torus $T^2=S^1\times S^1$
with the two dimensional sphere $S^2$ is not Jordan.
This is the first known
example of a compact manifold whose diffeomorphism group is not Jordan.
Previously, Popov \cite{Po1} had given an example of a connected open
$4$-manifold with non Jordan diffeomorphism group.
In contrast, there are many examples of manifolds whose diffeomorphism group is
known to be Jordan: these include all compact manifolds of dimension at most $3$,
all compact manifolds with nonzero Euler characteristic, homology spheres, the connected
sum of a torus and an arbitrary compact connected manifold,
and open contractible manifolds, see \cite{M1,M2,M3,M4,Z}.

Denote by $M_d$
the projectivisation of the complex vector bundle $L_d\oplus\un{\CC}\to T^2$,
were $L_d\to T^2$ is a degree $d$ line bundle and $\un{\CC}\to T^2$ is the trivial line bundle.
Csik\'os, Pyber and Szab\'o base their proof in two facts. First, $M_d$ is diffeomorphic
to $T^2\times S^2$ for any even $d$. Second, for positive $d$ there is a finite group $\Gamma_d$
of order $d^3$ acting effectively on $M_d$ such that any abelian subgroup $A$ of
$\Gamma_d$ has at most $d^2$ elements. The group $\Gamma_d$ is a Heisenberg group,
and its action on $M_d$ is induced by an effective linear action on $L_d$.
Picking algebraic structures on
$T^2$ and the vector bundle $L_d\oplus\un{\CC}$ the action of $\Gamma_d$ can
be taken to be algebraic. This is the key ingredient in a paper of
Zarhin \cite{Za} that gives the first example of algebraic manifold whose group of
birational transformations is not Jordan.

A slightly different way to present the arguments in \cite{CPS}
is the following. For any (non necessarily even) integer $d$
the complex vector bundle $V:=L_d\oplus L_d^{-1}\to T^2$ has
degree $0$. By the classification of complex vector bundles
over compact connected surfaces this implies that $V$ can be
trivialized. Furthermore one can pick a $\Gamma_d$-invariant
trivialization of $\det V$, because the action of $\Gamma_d$
naturally induced on $\det V$ factors through a free action of
$\ZZ_d^2$. Using a $\Gamma_d$-invariant Hermitian metric on $V$
with respect to which the section of $\det V$ defining the
trivialization has constant norm equal to $1$, the bundle $E$
of unitary frames compatible with the trivialization of $\det
V$ turns out to be a (trivial) $\SU(2)$ principal bundle, and
the action of $\Gamma_d$ on $L_d$ gives an effective action on
$E$. 

Now we can identify $T^2\times S^2$ with
$E\times_{\SU(2)}S^2$, where $\SU(2)$ acts on $S^2$ via the
quotient map $\SU(2)\to\SO(3,\RR)$ and the identification of
$\SO(3,\RR)$ with the orientation preserving isometries of
$S^2$. If $d$ is odd, the action of $\Gamma_d$ on $E$
induces an effective action on $E\times_{\SU(2)}S^2=T^2\times
S^2$. The parity restriction on $d$ is a consequence of the
fact that there is a subgroup $\{\pm 1\}\subset\SU(2)$ acting
trivially on $S^2$. Hence this point of view gives a slightly
less general construction than \cite{CPS}, since the latter
 allows to construct effective actions of $\Gamma_d$ on
$T^2\times S^2$ for even $d$; this discrepancy comes from the
fact that $E\times_{\SU(2)}S^2$ can be identified in a
$\Gamma_d$-equivariant way with
$\PP(L_d^2\oplus\underline{\CC})$. On the other hand, this
point of view immediately suggests that $S^2$ can be replaced
by any manifold admitting an effective action of $\SU(2)$ or
$\SO(3,\RR)$. For example, since any sphere of dimension at 
least $2$ supports such actions, we deduce that 
the diffeomorphism group of the product
$T^a\times S^b$ of a torus and a sphere of dimensions $a\geq 2$
and $b\geq 2$ is not Jordan.

Before stating our main theorem we introduce some terminology
and conventions. We define the kernel of the action of a group
$G$ on a space $X$ to be the subgroup of $G$ consisting of
those elements that act trivially on $X$. We say that the
action of $G$ on $X$ is almost effective if its kernel is
finite. All manifolds and actions of groups on manifolds which
we consider are implicitly assumed to be smooth. In this paper
by a natural number we mean a strictly positive integer. The
set of natural numbers is denoted as usually by $\NN$. For any
$k\in\NN$ we denote by $T^k$ the $k$-dimensional torus
$(S^1)^k$.

For any pair $(\tau,r)\in\NN^2$ we denote by $\SU(\tau)^r$ the
direct product of $r$ copies of the special unitary group
$\SU(\tau)$.

We are now ready to state the main result of this paper.

\begin{theorem}
\label{thm:main} There exist functions $\tau,M:\NN\to\NN$ such
that for any $n,r\in \NN$ the following property is satisfied. Suppose that
a manifold $X$ supports an almost effective action of
$\SU(\tau(n))^r$ with kernel $H$. For any prime $p$ not
dividing $|H|$ and satisfying $p>M(n)$ and $p\equiv 1\mod n+1$,
there exists a finite $p$-group $\Gamma$ acting effectively on
$T^{2nr}\times X$ such that $|\Gamma|=p^{2n+r}$ and:
\begin{enumerate}
\item if $r=1$ then any abelian subgroup of $\Gamma$ has at
    most $p^{n+1}$ elements;
\item if $r>1$ then any abelian subgroup of $\Gamma$ has at
    most $p^{2+r+4n/r}$ elements.
\end{enumerate}
\end{theorem}

The manifold $X$ need not be compact.

By Dirichlet theorem (see e.g. \cite[\S8.4]{NZM}), for any
$n,h\in\NN$ there are infinitely many primes $p$ that do not
divide $h$ and satisfy $p>M(n)$ and $p\equiv 1\mod n+1$. Hence,
once we fix $X$ and an almost effective action of
$\SU(\tau(n))^r$ on $X$ the previous theorem applies to
infinitely many primes.

Our motivation to state the result referring to $p$-groups
comes from the main theorem in \cite{MT}, according to which to
test whether the diffeomorphism group of a manifold is Jordan
it suffices to consider finite subgroups of $G$ whose cardinal
is divisible by at most two different primes (actually the
result in \cite{MT} applies more generally to any group $G$
admitting a constant $R$ such that any elementary $p$-group
contained in $G$ has rank at most $R$, for any prime $p$;
diffeomorphism groups have always this property, by a theorem
of Mann and Su \cite{MS}). A priori the diffeomorphism group of
a manifold might fail to be Jordan but still satisfy Jordan's
property restricted to $p$-groups, and Theorem \ref{thm:main}
makes it clear that if $X$ satisfies the hypothesis of the
theorem then the diffeomorphism group of $T^{2nr}\times X$ does
not even have this property.

The $p$-groups obtained in the proof of Theorem \ref{thm:main}
are all $2$-step nilpotent. It seems an interesting question to
explore whether there are compact smooth manifolds admitting
actions of $k$-step nilpotent $p$-groups for some $k\geq 3$ and
arbitrarily large primes $p$.

Given a nontrivial finite group $\Gamma$ define
$$\lambda(\Gamma):=\max\left\{\frac{\log|A|}{\log|\Gamma|}\mid
A\text{ abelian subgroup of }\Gamma\right\},$$ and for any group
$G$ containing arbitrarily big finite subgroups consider the
quantity
$$\Lambda(G):=\inf\{\lambda\mid \exists \{\Gamma_i\}_{i\in\NN},\text{ each $\Gamma_i$ is a finite subgroup of $G$, }|\Gamma_i|\to\infty,\,\lambda(\Gamma_i)\to \lambda\};$$
if the size of the finite subgroups of $G$ is uniformly bounded, then define $\Lambda(G)=1$.
In particular, if $\Lambda(G)<\epsilon$ for some $\epsilon>0$, then $G$ contains arbitrarily large
finite subgroups $\Gamma$ all of whose abelian subgroups have size at most $|\Gamma|^{\epsilon}$.

Obviously, $\Lambda(G)\in [0,1]$ for any $G$. If $\Lambda(G)<1$
then $G$ is not Jordan, and the difference $1-\Lambda(G)$ gives
some measure of how far $G$ is from being Jordan. The main
result in \cite{CPS} implies that $\Lambda(\Diff(T^2\times
S^2))\leq 2/3$, while our theorem implies that if $X$ supports
an almost effective action of $\SU(\tau(n))$ then
$\Lambda(\Diff(T^{2n}\times X))\leq (n+1)/(2n+1)$ for any $n$.
Moreover, if $X$ supports an almost effective action of the
product of $\SU(\tau(n))^r$ then $\Lambda(\Diff(T^{2nr}\times
M))\leq (2+r+4n/r)/(2n+r)$ for any $n$. In particular, for any
$\epsilon>0$ there exist manifolds $Y$ such that
$\Lambda(\Diff(Y))<\epsilon$.

We next describe the main building block in the proof of
Theorem \ref{thm:main}. For any natural number $n$ and any
prime $p$, define $\Gamma_{n,p}$ to be the group generated by
elements $a_1,\dots,a_n,b_1,\dots,b_n,f$ with the relations
$a_i^p=b_i^p=f^p=[a_i,a_j]=[b_i,b_j]=[a_i,f]=[b_i,f]=1$ for
every $i,j$, $[a_i,b_j]=1$ for very $i\neq j$, and
$[a_i,b_i]=f$ for every $i$. The group $\Gamma_{n,p}$ has
$p^{2n+1}$ elements and no abelian subgroup of $\Gamma_{n,p}$
has more than $p^{n+1}$ elements (see Lemma
\ref{lemma:abelian-subgroups-Gamma-n-p} below). We have:

\begin{theorem}
\label{thm:vector-bundles}
Given $n\in\NN$ there exists some $\tau(n),M(n)\in\NN$ such that
for any
prime $p$ satisfying $p>M(n)$ and
$p\equiv 1\mod n+1$ the group $\Gamma_{n,p}$ acts effectively on
the trivial vector bundle $T^{2n}\times\CC^{\tau(n)}$ by vector bundle automorphisms
and leaving invariant a nowhere vanishing section of the determinant bundle
$T^{2n}\times\Lambda^{\tau(n)}\CC^{\tau(n)}$.
\end{theorem}

To deduce Theorem \ref{thm:main} from Theorem
\ref{thm:vector-bundles} we use some standard constructions of
fiber bundles and a group theoretical result of Olshanskii
\cite{O}.

As explained above, when $n=1$ Theorem \ref{thm:vector-bundles}
follows from the fact that for any degree $d$ the vector bundle
$L_d\oplus L_d^{-1}\to T^2$ is trivial as a smooth vector
bundle. Hence, we may take $M(1)=1$ and $\tau(1)=2$. An
immediate consequence of Theorem \ref{thm:main} is that if a
manifold $X$ has the property that $T^{2}\times X$ is Jordan
then $X$ does not support any almost effective action of
$\SU(2)$. Since any compact connected semisimple Lie group
contains a subgroup isomorphic either to $\SU(2)$ or to
$\SO(3,\RR)\simeq\SU(2)/\{\pm\Id\}$ (see for example Theorem
19.1 in \cite{B}), it follows that $X$ does not admit any
effective action of a compact connected semisimple Lie group.
In view of the main theorem in \cite{M1} this implies the
following.

\begin{corollary}
\label{cor:main} Suppose that $X$ is a $d$-dimensional compact
manifold admitting a finite unramified covering $\wt{X}\to X$
and that there exist classes $\alpha_1,\dots,\alpha_d\in
H^1(\wt{X};\ZZ)$ such that $\alpha_1\cup\dots\cup \alpha_d\neq
0$. Then any compact connected Lie group acting effectively on
$X$ is abelian.
\end{corollary}

This applies in particular to the connected sum of a torus and any other manifold.
Corollary \ref{cor:main}
is not a new result (see \cite[Theorem 2.1]{DS}, \cite[Theorem A]{WW};
see also  \cite[Theorem 2.5]{GLO}, which is slightly more restrictive), but the
proof we obtain is new.

\subsection{Acknowledgements} I wish to thank Artur Travesa for useful conversations on
Proposition 2.7.

\section{Proof of Theorem \ref{thm:vector-bundles}}

Fix some odd prime $p$ and a natural number $n$.

\subsection{The group $\Gamma_{n,p}$ and the bundle $Q_{n,p}$ over the torus $T^{2n}$.}
\label{ss:Heisenberg}
We begin by reviewing the construction of some standard generalizations of Heisenberg
$p$-groups and their action on bundles over $T^{2n}$.
Let $$X=\RR^n\times T^n\times S^1.$$ Define a free action
of $\ZZ^n$ on $X$ by setting, for any $d=(d_1,\dots,d_n)\in\ZZ^n$,
$t\in\RR^n$, $\theta=(\theta_1,\dots,\theta_n)\in T^n$, and $\nu\in S^1$:
\begin{equation}
\label{eq:action-Z-n}
d\cdot (t,\theta,\nu)=\left(t+d,\theta,\theta_1^{d_1p}\dots\theta_n^{d_np}\nu\right).
\end{equation}
Denote by $Q_{n,p}=X/\ZZ^n$ the orbit space of this action. The projection of $X$ to the first
two factors gives $Q_{n,p}$ a structure of principal $S^1$-bundle over $T^{2n}$.

Let $\mu=\exp(2\pi\imag/p)$ and let $\psi_j=(1,\dots,1,\mu,1,\dots,1)\in T^n$, where
the entry $\mu$ is in the $j$-th position. Let $e_1,\dots,e_n$ be the canonical
basis of $\RR^n$. Define
diffeomorphisms $\alpha_1,\dots,\alpha_n,\beta_1,\dots,\beta_n,\phi\in\Diff(X)$ by the
formulas
$$\alpha_j(t,\theta,\nu) = (t+p^{-1}e_j,\theta,\theta_j\nu),\qquad
\beta_j(t,\theta,\nu) = (t,\psi_j\theta,\nu),\qquad
\phi(t,\theta,\nu) = (t,\theta,\mu\nu).
$$
Let $\Gamma$ be the group generated by
$\alpha_1,\dots,\alpha_n,\beta_1,\dots,\beta_n,\phi$. One
easily checks that
$[\alpha_i,\alpha_j]=[\beta_i,\beta_j]=\beta_i^p=1$ for every
$i,j$, that $[\alpha_i,\beta_j]=1$ for every $i\neq j$, and
that $[\alpha_j,\beta_j]=\phi$ for every $j$. Furthermore,
$\phi$ is central in $\Gamma$ and has order $p$, so
$[\alpha_j^p,\beta_j]=1$ for every $j$. Hence each $\alpha_j^p$
is central in $\Gamma$. Let $\Gamma_Z\subseteq \Gamma$ be the
subgroup generated by $\alpha_1^p,\dots,\alpha_r^p$. The action
of $\Gamma_Z$ on $X$ coincides with the action of $\ZZ^n$
defined in (\ref{eq:action-Z-n}): more precisely, we can
identify the diffeomorphism $\alpha_j^p$ with the action of
$e_j\in\ZZ^n$. So the quotient group
$\Gamma_{n,p}=\Gamma/\Gamma_Z$ acts effectively on $Q_{n,p}$
The group $\Gamma_{n,p}$ is obviously the same as the one
defined before the statement of Theorem
\ref{thm:vector-bundles}.

Denote by $\ov{\alpha}_j$, $\ov{\beta}_j$ and $\ov{\phi}$ the classes in $\Gamma_{n,p}$
of the elements $\alpha_j$, $\beta_j$ and $\phi$. The
group $\Gamma_{n,p}$ sits in a short exact sequence
\begin{equation}
\label{eq:exact}
0\to\ZZ_p\to\Gamma_{n,p}\stackrel{\eta}{\longrightarrow}(\ZZ_p)^{2n}\to 0,
\end{equation}
where $\eta(\ov{\alpha}_j)\in\ZZ_p^{2n}$ (resp. $\eta(\ov{\beta}_j)$)
is the tuple with $0$'s everywhere except in the position
$j$ (resp. $r+j$), where the entry is $1$ (consequently, $\eta(\ov{\phi})=0$).
Hence $\Gamma_{n,p}$ has $p^{2n+1}$ elements.
In terms of the standard
symplectic form $\omega:\ZZ_p^{2n}\times \ZZ_p^{2n}\to\ZZ_p$ defined by
\begin{equation}
\label{eq:omega}
\omega((x_1,\dots,x_n,y_1,\dots,y_n),(x_1',\dots,x_n',y_1',\dots,y_n'))=\sum (x_jy_j'-x_j'y_j)
\end{equation}
we have $[\zeta,\gamma]=\phi^{\omega(\eta(\zeta),\eta(\gamma))}$ for any $\zeta,\gamma\in\Gamma_{n,p}$.
Hence any abelian subgroup of $\Gamma_{n,p}$ projects via $\eta$ to an $\omega$-isotropic
subspace of $\ZZ_p^{2n}$. Since $\omega$ is non degenerate, $\omega$-isotropic subspaces of $\ZZ_p^{2n}$ have dimension at
most $n$, so their cardinal is at most $p^n$. We have thus proved the following.

\begin{lemma}
\label{lemma:abelian-subgroups-Gamma-n-p}
No abelian subgroup of $\Gamma_{n,p}$
has more than $p^{n+1}$ elements.
\end{lemma}

The action of $\Gamma_{n,p}$ on $Q_{n,p}$ lifts an action of
$\Gamma_{n,p}$ on $T^{2n}$ which is not effective, since two
elements $\gamma,\gamma'\in\Gamma_{n,p}$ act via the same
diffeomorphism of $T^{2n}$ if and only if
$\eta(\gamma)=\eta(\gamma')$. Hence the action of
$\Gamma_{n,p}$ induces an effective action of $\ZZ_p^{2n}$ on
$T^{2n}$, which of course is nothing but the diagonal action
\begin{equation}
\label{eq:accio}
(z_1,\dots,z_{2n})\cdot(\theta_1,\dots,\theta_{2n})=
(e^{2\pi\imag z_1/p}\theta_1,\dots,e^{2\pi\imag z_{2n}/p}\theta_{2n}).
\end{equation}
The quotient space of this action of $\ZZ_p^{2n}$ on $T^{2n}$
can be identified with $T^{2n}$ itself, in such a way that the projection to the quotient space is the map
\begin{equation}
\label{eq:projection-quotient}
w:T^{2n}\to T^{2n},\qquad w(\theta_1,\dots,\theta_{2n})=(\theta_1^p,\dots,\theta_{2n}^p).
\end{equation}

\begin{lemma}
\label{lemma:accio-fibrat-trivial} Suppose that $\Gamma_{n,p}$
acts on the trivial line bundle $T^{2n}\times\CC$ by vector
bundle automorphisms lifting the action of $\Gamma_{n,p}$ on
$T^{2n}$. Then there is a nowhere vanishing
$\Gamma_{n,p}$-equivariant section $\sigma:T^{2n}\to\CC$.
\end{lemma}
That $\sigma$ is an  equivariant section means that $\gamma\cdot (\theta,\sigma(\theta))=(\gamma\cdot\theta,\sigma(\gamma\cdot\theta))$
for every $\gamma\in\Gamma_{n,p}$ and every $\theta\in T^{2n}$.

\begin{proof}
Let $L=T^{2n}\times\CC$.
Take an action of $\Gamma_{n,p}$ on $L$ by vector
bundle automorphisms lifting the action of $\Gamma_{n,p}$ on $T^{2n}$.
There is a unique smooth map $c:\Gamma_{n,p}\times T^{2n}\to\CC^*$ satisfying
$$\gamma\cdot(\theta,w)=(\gamma\cdot\theta,c(\gamma,\theta)w)$$
for every $\gamma\in\Gamma_{n,p}$, $\theta\in T^{2n}$ and
$w\in\CC$. The condition that $c$ defines an action of
$\Gamma_{n,p}$ on $L$ is equivalent to the following cocycle
condition:
$$c(\gamma'\gamma,\theta)=c(\gamma',\gamma\cdot\theta)c(\gamma,\theta)
\qquad\qquad\text{for every $\gamma,\gamma'\in\Gamma_{n,p}$ and
$\theta\in T^{2n}$.}$$ This implies in particular that
$c(1,\theta)=1$ for every $\theta$, where $1\in\Gamma_{n,p}$
denotes the identity element. We claim that for any
$\gamma\in\Gamma_{n,p}$ there is a map
$$\wt{c}_{\gamma}:T^{2n}\to\CC$$
such that $c(\gamma,\theta)=\exp(\wt{c}_{\gamma}(\theta))$ for
every $\theta$. This is obvious if $\gamma=1$, because
$c(1,\theta)=1$. So let us assume that $\gamma\neq 1$ and that,
contrary to the claim, the map $\wt{c}_{\gamma}$ does not
exist. Then there exists some $\gamma\in\Gamma_{n,p}$ and a map
$h:S^1\to T^{2n}$ so that the map $c_{\gamma,h}:S^1\to\CC^*$
defined as $c_{\gamma,h}(\chi)=c(\gamma,h(\chi))$ has nonzero
index: $\Ind(c_{\gamma,h})\neq 0$. For any
$\delta\in\Gamma_{n,p}$ let $\mu_{\delta}:T^{2n}\to T^{2n}$ be
the map $\theta\mapsto\delta\cdot\theta$. Since $\mu_{\delta}$
is homotopic to the identity, we have
\begin{equation}
\label{eq:invariancia-index}
\Ind(c_{\gamma,\mu_{\delta}\circ h})=\Ind(c_{\gamma,h}).
\end{equation}
Using the cocycle condition and induction on $k\in\NN$ we obtain
$$c(\gamma^k,\theta)=\prod_{j=0}^{k-1}c(\gamma,\gamma^j\cdot\theta)$$
for every $\gamma$ and $k$. Taking $k=\ord(\gamma)+1$ (so that
$\gamma^k=\gamma$), applying the previous formula to
$\theta=h(\chi)$ for each $\chi\in S^1$, and using the fact
that the index of maps $S^1\to\CC^*$ is additive with respect
to pointwise multiplication, we deduce
$$\Ind(c_{\gamma,h})=\sum_{j=0}^{k-1}\Ind(c_{\gamma,\mu_{\gamma^j}\circ h})=k\Ind(c_{\gamma,h}),$$
where the second equality follows from
(\ref{eq:invariancia-index}). Since $\gamma\neq 1$, we have
$k\geq 2$ and hence $\Ind(c_{\gamma,h})=0$, contrary to our
assumption. So the claim is proved.

Now let $a=\ov{\alpha}_1$, $b=\ov{\beta}_1$ and $f=\ov{\phi}$.
We next  claim  that the action of $f$ on $L$ is trivial (so in
particular the action of $\Gamma_{n,p}$ is not effective). In
order to prove the claim, note that since $f$ acts trivially on
$T^{2n}$, its action on $L$ will be given by
$f\cdot(\theta,w)=(\theta,g(\theta)w)$ for some smooth map
$g:T^{2n}\to\CC^*$. Since $f$ has order $p$, $g$ must take
values in the set of $p$-roots of unity. Applying the cocycle
condition to both sides of the equality $ab=baf$ we obtain
$$c(a,b\cdot\theta)c(b,\theta)=c(ba,f\cdot\theta)c(f,\theta)=
c(ba,\theta)g(\theta)=c(b,a\cdot\theta)c(a,\theta)g(\theta)$$
for any $\theta$. It follows that
\begin{equation}
\label{eq:cocycle-commutador}
\wt{c}_a(b\cdot\theta)+\wt{c}_b(\theta)=
\wt{c}_b(a\cdot\theta)+\wt{c}_a(\theta)+\wt{g}(\theta),
\end{equation}
where $\theta$ is arbitrary and $\wt{g}:T^{2n}\to\CC$ satisfies
$\exp(\wt{g}(\theta))=g(\theta)$, so
$\exp(p\wt{g}(\theta))=g(\theta)^p=1$ for every $\theta$. Since
each $\wt{c}_{\gamma}$ is smooth, (\ref{eq:cocycle-commutador})
implies that $\wt{g}$ is smooth, and hence the condition
$\exp(p\wt{g}(\theta))=1$ implies that $\wt{g}$ must be equal
to some constant, say $\wt{g}_0\in\CC$. Now let $q\in T^{2n}$
be any point and let $\Theta=\Gamma_{n,p}\cdot
q=\ZZ_p^{2n}\cdot q\subset T^{2n}$ be its orbit. Summing both
sides of (\ref{eq:cocycle-commutador}) as $\theta$ runs over
the elements of $\Theta$ and using the fact that $\Theta$ is
invariant under the action of both $a$ and $b$ we deduce that
$0=\sum_{\theta\in\Theta}\wt{g}(\theta)=|\Theta|\wt{g}_0$.
Hence $\wt{g}_0=0$, so $g(\theta)=1$ for every $\theta$ and the
claim is proved.

The previous claim implies that the action of $\Gamma_{n,p}$ on
$L$ factors through the morphism
$\eta:\Gamma_{n,p}\to\ZZ_p^{2n}$. Since $\ZZ_p^{2n}$ acts
freely on $T^{2n}$, the quotient $L/\ZZ_p^{2n}$ has a natural
structure of line bundle over $T^{2n}/\ZZ_p^{2n}$. In other
words, there must exist some line bundle $\Lambda\to T^{2n}$
and a $\Gamma_{n,p}$-equivariant isomorphism $L\simeq
w^*\Lambda$, where $w$ is the map
(\ref{eq:projection-quotient}). This implies that
$c_1(L)=w^*c_1(\Lambda)$. But $w^*:H^2(T^{2n};\ZZ)\to
H^2(T^{2n};\ZZ)$ is multiplication by $p^2$, and hence is
injective. Since $c_1(L)=0$, it follows that $c_1(\Lambda)=0$.
Hence $\Lambda$ is the trivial line bundle. Taking a nowhere
vanishing section $s$ of $\Lambda$, we obtain by pullback the
desired nowhere vanishing $\Gamma_{n,p}$-equivariant section
$\sigma$ of $L$.
\end{proof}

\subsection{Cohomology of $T^{2n}$} Denote the line bundle associated to $Q_{n,p}$ by
$$L_{n,p}=Q_{n,p}\times_{S^1}\CC,$$
where $S^1$ acts on $\CC$ with weight $1$. Choose a generator
$\sigma\in H^1(S^1;\ZZ)$. Let $\pi_j:T^{2n}=(S^1)^{2n}\to S^1$
denote the projection to the $j$-th factor and define
cohomology classes $u_i,v_i\in H^1(T^{2n};\ZZ)$ for any $1\leq
i\leq n$ by $u_i=\pi_i^*\sigma$ and $v_i=\pi_{n+i}^*\sigma$.
Define also the class
$$\Omega=\sum_{i=1}^n u_i\cup v_i\in H^2(T^{2n};\ZZ).$$

\begin{lemma}
We have $c_1(L_{n,p})=p\Omega$.
\end{lemma}
\begin{proof}
Let $\Pi_i=(\pi_{i},\pi_{n+i}):T^{2n}\to S^1\times S^1=T^2$. It
follows from the definition of $Q_{n,p}$ that
$L_{n,p}=\bigotimes_{i=1}^n\Pi_i^*L_{1,p}.$ Hence it suffices
to prove the lemma in the case $n=1$. Using the notation of
Subsection \ref{ss:Heisenberg} we identify $Q_{1,p}$ with the
quotient of the trivial principal bundle
$$\pi:\wt{Q}_{1,p}=\RR\times S^1\times S^1\to\RR\times S^1,$$
where $\pi$ is the projection to the first two factors, under
the action of $\ZZ$ given by
$k\cdot(t,\theta,\nu)=(t+k,\theta,\theta^{pk}\nu)$. Then
$d_A:=d-\imag p t\,d\theta$ is a $\ZZ$-invariant connection, so
it descends to a connection on $Q_{1,p}$. Its curvature on
$\wt{Q}_{1,n}$ is equal to $-\imag p\,dt\wedge d\theta$, so we
may compute, using the orientation of $T^{2n}$ given by
$dt\wedge d\theta$
$$\deg Q_{1,n}=\frac{\imag}{2\pi}\int_0^1\left(\int_{S^1}-\imag p\,d\theta\right)\,dt=p.$$
Since the integral of $\Omega$ with respect to the same orientation is equal to $1$,
the result follows.
\end{proof}

For any $m\in\NN$ we define $[m]:=\{1,\dots,m\}$.
For any subset $I=\{i_1<\dots<i_k\}\subset[n]$
define $u_I=u_{i_1}\cup\dots \cup u_{i_k}$ and
$v_I=v_{i_1}\cup\dots \cup v_{i_k}$. A simple computation shows that
\begin{equation}
\label{eq:power-Omega}
\Omega^k=\frac{n!}{(n-k)!}\sum_{I\subset[n],\,|I|=k}(-1)^k u_I\cup v_I.
\end{equation}

Given a permutation $\sigma\in S_n$ define a new permutation
$\sigma'\in S_{2n}$ by the condition $\sigma'(i)=\sigma(i)$ and
$\sigma'(n+i)=n+\sigma(i)$ for very $1\leq i\leq n$. Let
$\nu_{\sigma}:T^{2n}\to T^{2n}$ be the diffeomorphism defined
by
$$\nu_{\sigma}(\theta_1,\dots,\theta_{2n})=
(\theta_{\sigma'(1)},\dots,\theta_{\sigma'(2n)}).$$
For any real number $t$ denote by $\la t\ra$ the integer part of $t$.

\begin{lemma}
\label{lemma:symmetrization} For any $k\in [n]$ there exist
rational numbers $\{a_{k,1},\dots,a_{k,\la n/k\ra}\}$, where
each $a_{k,j}$ depends on $k,j,n$ but not on $p$, such that
$$\prod_{\sigma\in S_n}(1+\nu_\sigma^*(u_{[k]}\cup v_{[k]}))=
1+\sum_{j=1}^{\la n/k\ra}a_{k,j}\Omega^{jk}.$$
Furthermore, $a_{k,1}\neq 0$.
\end{lemma}
\begin{proof}
The formula in the lemma follows from (\ref{eq:power-Omega}).
The nonvanishing of $a_{k,1}$ is a consequence of the fact that
if for some $\sigma\in S_n$ we have $\sigma([k])=[k]$ then
$\nu_{\sigma}^*(u_{[k]}\cup v_{[k]})=u_{[k]}\cup v_{[k]}$ i.e.,
there is no minus sign for any choice of $\sigma$.
\end{proof}

\subsection{Some equivariant vector bundles over $T^{2n}$}

\begin{lemma}
\label{lemma:bundle-over-sphere} Let $\sigma_k\in
H^k(S^{2k};\ZZ)$ be a generator. For any natural number $k$ and
any integer $\delta$ there exists a complex vector bundle
$E_k(\delta)\to S^{2k}$ of rank $k$ satisfying
$c_k(E_k(\delta))=\delta(k-1)!\sigma_k$,
\end{lemma}
\begin{proof}
By \cite[Ch. 20, Corollary 9.8]{H}
there exists some class $\epsilon(\delta)\in \wt{K}(S^{2k})$ with $c_k(\epsilon(\delta))=\delta(k-1)!\sigma_k$.
By \cite[Ch. 9, Theorem 3.8]{H} there exists a vector bundle $\xi\to S^{2k}$ such that $\epsilon(\delta)=[\xi]-[(\rk\xi)\un{\CC}]$.
Finally, by \cite[Ch. 9, Remark 3.7]{H}, the vector bundle
$\xi$ is stably equivalent to a vector bundle on $S^{2k}$ of rank $k$.
We take $E_k(\delta)$ to be any such vector bundle.
\end{proof}

\begin{lemma}
\label{lemma:bundle-over-torus}
For any $k\in [n]$ and any $\delta\in\ZZ$ there exists a vector bundle
$F_k^0(\delta)$ of rank $k$ over $T^{2n}$ satisfying
$$c_k(F_k^0(\delta))=\delta(k-1)!(u_{[k]}\cup v_{[k]})$$
and $c_j(F_k^0(\delta))=0$ for any $j\neq k$.
\end{lemma}
\begin{proof}
Let $\pi^k:T^{2n}\to T^{2k}$ denote the projection
$$(\theta_1,\dots,\theta_n,\theta_{n+1},\dots,\theta_{2n})
\mapsto
(\theta_1,\dots,\theta_k,\theta_{n+1},\dots,\theta_{n+k})$$
and
let $q_k:T^{2k}\to S^{2k}$ be a degree $\pm 1$ map. We may
suppose that the generator $\sigma_k\in H^k(S^{2k};\ZZ)$ in
Lemma \ref{lemma:bundle-over-sphere} satisfies
$q_k^*\sigma_k=u_{[k]}\cup v_{[k]}$. Then we define
$F_k^0(\delta)=(\pi^k)^*q_k^*E_k(\delta)$, where $E_k(\delta)$
is any bundle as given by Lemma \ref{lemma:bundle-over-sphere}.
\end{proof}

Fix for any $k\in [n]$ and any integer $\delta$ a vector bundle $F_k^0(\delta)$
over $T^{2n}$ with the properties specified in Lemma \ref{lemma:bundle-over-torus}.
Define, for any $k$ and $\delta$,
$$F_k(\delta):=\bigoplus_{\sigma\in S_n}\nu_{\sigma}^*F_k^0(\delta).$$
The vector bundle $F_k(\delta)$ has rank $kn!$.
Lemma \ref{lemma:symmetrization} implies that the total Chern class of $F_k(\delta)$ is
$$c(F_k(\delta))=1+\sum_{j=1}^{\la n/k\ra}\delta^ja_{k,j}\Omega^{jk},$$
with $a_{k,1}\neq 0$.

\begin{lemma}
\label{lemma:spanning-classes}
There exists $M\in\NN$, depending only on $n$, with the following property.
Suppose that $b_1,\dots,b_n\in\ZZ$ are all divisible by $M$.
Then there exist $\delta_1,\dots,\delta_n\in\ZZ$ 
with the property that
$$\prod_{j=1}^n c(F_j(\delta_j))
=1+\sum_{j=1}^n b_j\Omega^j.$$
\end{lemma}
\begin{proof}
Denote for any $k\in\NN$ satisfying $k\leq n$ and any $m\in\ZZ$
$$\hH^{\geq k}(m)=\{\alpha\in H^*(T^{2n};\ZZ)\mid \alpha=1+\alpha_k\Omega^k+\dots+\alpha_n\Omega^n,
\,a_j\in m\ZZ\text{ for each $j$}\}.$$
We claim that for any $1\leq k\leq n$ there exists some integer
$m_k$ such that any $\alpha\in\hH^{\geq k}(m_k)$ can be written as
$\prod_{j=k}^n c(F_j(\delta_j))$ for a suitable choice of $\delta_k,\dots,\delta_n\in\ZZ$.
Of course the case $k=1$ is the lemma we want to prove.

We prove the claim by descending induction on $k$.
Consider first the case $k=n$.
Choose $m_n$ so that any element of $m_n\ZZ$ is an integral multiple of $a_{n,1}$.
Since for any integer $b_n$ we have $1+m_nb_n\Omega^j=c(F_n(m_nb_n/a_{n,1}))$ and
$m_nb_n/a_{n,1}$ is an integer, we are done in this case.
Now assume that the claim has been proved for some $k=i+1$ ($1\leq i<n$)
and an integer $m_{i+1}$.
Let $m_i'\in\ZZ$ be chosen in such a way that any element of $m_i'\ZZ$ is an integral
multiple of $a_{i,1}$, let $m_i''\in\ZZ$ be chosen in such a way that
$m_i''a_{i,j}\in m_{i+1}\ZZ$ for every $1<j$, and let $m_i:=m_i'm_i''$.
We next prove that this choice of $m_i$ has the desired property.

If $\alpha=1+\alpha_i\Omega^i+\dots+\alpha_n\in\hH^{\geq i}(m_i)$ then
$\delta_i:=\alpha_i/a_{i,1}$ belongs to $m_i''\ZZ$. Consequently, in the development
$c(F_i(\delta_i))=1+\sum_{j\geq 1}^{\la n/i\ra}\gamma_j\Omega^{ij}$ we have
$\gamma_j\in m_{i+1}\ZZ$ for every $j>1$.
This implies that the series
$\alpha'=1+\sum_{j\geq 1}\alpha_j'\Omega^j$, defined by the property that
$\alpha=c(F_i(\delta_i))\alpha'$, belongs to $\hH^{\geq i+1}(m_{i+1})$. By the induction hypothesis we may write $\alpha'=\prod_{j=i+1}^n c(F_j(\delta_j))$ for some $\delta_{i+1},\dots,\delta_n\in\ZZ$. Hence $\alpha=\prod_{j=i}^n c(F_j(\delta_j))$,
so the claim is proved.
\end{proof}

Let $w:T^{2n}\to T^{2n}$ be the map (\ref{eq:projection-quotient}).
Define for any $k$ and $\delta$
$$G_k(\delta):=w^*{F_k(\delta)}.$$
Since the fibers of $w$ are the orbits of the action of $\ZZ_p^{2n}$ on $T^{2n}$,
the vector bundle $F_k$ carries a natural action of $\ZZ_p^{2n}$ lifting the action on
$T^{2n}$. The action of $\ZZ_p^{2n}$ on $G_k(\delta)$ can be promoted to an action of
$\Gamma_{n,p}$ via the projection map $\eta:\Gamma_{n,p}\to\ZZ_p^{2n}$.
Of course, this action of $\Gamma_{n,p}$ is not effective.

Applying K\"unneth
it follows that the morphism induced in cohomology by $w$ is
\begin{equation}
\label{eq:w-cohomology}
w^i:H^i(T^{2n};\ZZ)\to H^i(T^{2n};\ZZ),\qquad w^i(\alpha)=p^i\alpha.
\end{equation}
Hence we have
$$c(G_k(\delta))=1+\sum_{j=k}^{\la n/k\ra}\delta^j p^{2jk}a_{k,j}\Omega^{jk}.$$
The next lemma follows immediately from Lemma
\ref{lemma:spanning-classes}.

\begin{lemma}
\label{lemma:spanning-classes-G}
There exists $M\in\NN$, depending only on $n$, with the following property.
Suppose that $b_1,\dots,b_n\in\ZZ$ are such that $b_j$ is divisible by $Mp^{2j}$
for each $j$.
Then there exist $\delta_1,\dots,\delta_n\in\ZZ$ 
with the property that
$$\prod_{j=1}^n c(G_j(\delta_j))=1+\sum_{j=1}^n b_j\Omega^j.$$
\end{lemma}

\begin{prop}
\label{prop:main-computation}
Let $M$ be as in Lemma \ref{lemma:spanning-classes-G}.
Suppose that
$p\equiv 1\mod n+1$.
There exist integers
$a_1,\dots,a_{n+1},\delta_1,\dots,\delta_n\in\ZZ$
such that $p$ does not divide any of
the $a_j$'s and furthermore
\begin{equation}
\label{eq:inversa}
\prod_{j=1}^{n+1}(1+a_j M p\Omega)\prod_{j=1}^n c(G_j(\delta_j))=1.
\end{equation}
\end{prop}
\begin{proof}
Since $p\equiv 1\mod n+1$, the group of invertible elements
$(\ZZ_{p^n})^*$ in $\ZZ_{p^n}$ has order $(p-1)p^{n-1}$
divisible by $n+1$. Since $(\ZZ_{p^n})^*$ is cyclic
(equivalently, $(\ZZ_{p^n})^*$ has primitive roots, see e.g.
\cite[\S 2.8]{NZM}), this implies that $R=\{\alpha\in
(\ZZ_{p^n})^*\mid \alpha^{n+1}=1\}$ is a cyclic subgroup of
$(\ZZ_{p^n})^*$ of $n+1$ elements. Let
$R=\{\alpha_1,\dots,\alpha_{n+1}\}$, choose integers
$a_1,\dots,a_{n+1}$ such that $a_j\equiv \alpha_j\mod p^n$ for
each $j$, and define $s_1,\dots,s_n$ by the condition
$$\prod_{j=1}^{n+1}(1+a_j M p\Omega)=1+\sum_{j=1}^n s_j M^j p^j\Omega^j.$$
Clearly, $s_j$ can be identified with the $j$-th symmetric
function $\sigma_j(a_1,\dots,a_{n+1})$. We claim that each
$\sigma_j(a_1,\dots,a_{n+1})$ is divisible by $p^n$.
Equivalently, $\sigma_j(\alpha_1,\dots,\alpha_{n+1})=0$ in
$\ZZ_{p^n}$. Let $\alpha\in R$ be a generator. Since
multiplication by $\alpha$ induces a permutation of the elements of $R$, we
have, for any $1\leq j\leq n$,
$\alpha^j\sigma_j(\alpha_1,\dots,\alpha_{n+1})=
\sigma_j(\alpha\alpha_1,\dots,\alpha\alpha_{n+1})=
\sigma_j(\alpha_1,\dots,\alpha_{n+1}),$
which implies
\begin{equation}
\label{eq:permutacio}
(\alpha^j-1)\sigma_j(\alpha_1,\dots,\alpha_{n+1})=0.
\end{equation}
We next prove that $\alpha^j-1$ is an invertible element of
$\ZZ_{p^n}$. Since $\alpha$ is a generator of $S$ and
$j\leq n$, we have $\alpha^j\neq 1$. So if $\alpha^j-1$ were
not invertible then we could write $\alpha^j=1+\beta p^e$ for
some $1\leq e\leq n-1$ and some invertible $\beta\in\ZZ_{p^n}^*$. But
then $(\alpha^j)^{n+1}=1+(n+1)\beta p^e+\beta'p^{e+1}$ for some
$\beta'$. Since $p$ does not divide $n+1$, it follows that
$(n+1)\beta\in\ZZ_{p^n}^*$, so
$(\alpha^{n+1})^j=(\alpha^j)^{n+1}\neq 1$, a contradiction.
Finally, since $\alpha^j-1$ is invertible, (\ref{eq:permutacio}) implies
$\sigma_j(\alpha_1,\dots,\alpha_{n+1})=0$, which is what we
wanted to prove.

Let now $$s:=\sum_{j=1}^n s_j M^j p^j\Omega^j$$
and define integers $b_1,\dots,b_n$ by the condition that
$$\sum_{k\geq 1}(-1)^ks^k=\sum_{j=1}^n b_j\Omega^j.$$
It is easy to prove, using the fact that each $s_j$ is
divisible by $p^n$, that $b_j$ is divisible by $Mp^{2j}$ for
each $j$. By Lemma \ref{lemma:spanning-classes-G}, there exist
integers $\delta_1,\dots,\delta_n$ 
such that
$$\prod_{j=1}^n c(G_j(\delta_j))=1+\sum_{j=1}^n b_j\Omega^j.$$
Since
$1+\sum b_j\Omega^j=1+\sum_{k\geq 1}(-1)^ks^k$ is the inverse of $1+s=\prod(1+a_j M p\Omega)$,
the numbers $a_i$, $\delta_j$ and $e_k$ satisfy (\ref{eq:inversa}).
\end{proof}

\subsection{Trivial equivariant vector bundles on $T^{2n}$: proof of Theorem \ref{thm:vector-bundles}}

Let $\uC^r$ denote the trivial complex vector bundle of rank $r$ over the torus $T^{2n}$.

\newcommand{\ch}{\operatorname{ch}}

\begin{lemma}
\label{lemma:trivial-bundles}
There exists some number $r_0$ with the property that for any complex vector bundle
$V\to T^{2n}$ with vanishing Chern classes there is an isomorphism of vector bundles
$$V\oplus\uC^{r_0}\simeq \uC^{\rk V+r_0}.$$
\end{lemma}
\begin{proof}
We first claim that if a complex vector bundle over $T^{2n}$  has vanishing Chern classes
then it represents the trivial element in $K^0(T^{2n})$. This follows from combining
two facts. First, $K^*(T^{2n})$ has no torsion: this is a consequence of the isomorphisms
$K^0(S^1)\simeq K^{-1}(S^1)\simeq\ZZ$ (see \cite[Example 2.8.1]{P}) and K\"unneth theorem
for $K$-theory (see \cite[Proposition 3.3.15]{P}). The second fact is that if $X$ is
a topological space which is homeomorphic to a finite CW-complex and $K^*(X)$
is torsion free, then the Chern character $\ch:K^*(X)\to H^*(X;\QQ)$ is injective
(see the Corollary in \cite[\S 2.4]{AH}).
Since a vector bundle with vanishing Chern classes has trivial Chern character, the claim follows.
To deduce the lemma from the claim, apply \cite[Ch. 9, Theorem 1.5]{H}.
\end{proof}

We are now ready to prove Theorem \ref{thm:vector-bundles}.
Assume that $p\equiv 1\mod n+1$.
Let $M$ be as in Lemma \ref{lemma:spanning-classes-G}, and
let $a_1,\dots,a_{n+1},\delta_1,\dots,\delta_n\in\ZZ$
be as in
Proposition \ref{prop:main-computation}. Consider the following vector bundle
over $T^{2n}$:
$$V_{n,p}=\bigoplus_{j=1}^{n+1}L_{n,p}^{a_jM}\oplus
\bigoplus_{j=1}^n G_j(\delta_j).$$
By Proposition \ref{prop:main-computation} all Chern classes of $V_{n,p}$ are trivial.

The action of $\Gamma_{n,p}$ on $L_{n,p}$ induces actions on its powers $L_{n,p}^{a_jM}$.
Combining these actions with those on the bundles $G_j(\delta_j)$, we obtain
an action of $\Gamma_{n,p}$ on $V_{n,p}$ by vector bundle automorphisms.

\begin{lemma}
If $p>M$ then the action of $\Gamma_{n,p}$ on $V_{n,p}$ is effective.
\end{lemma}
\begin{proof}
It suffices to prove that the action of $\Gamma_{n,p}$ on any
of the summands $L_{n,p}^{a_jM}$ is effective. If for some $j$
there were a nontrivial element $\gamma\in\Gamma_{n,p}$ acting
trivially on $L_{n,p}^{a_jM}$ then its action on $T^{2n}$ would
be trivial, i.e., $\gamma\in\Ker \eta$, see the exact sequence
(\ref{eq:exact}). Any nontrivial element $\theta\in
\Ker\eta\simeq\ZZ_p$ acts on the circle bundle $Q_{n,p}$ via
the action of a nontrivial $p$-root of unity $\mu_{\theta}\in
S^1$. Then the action of $\theta$ on $L_{n,p}^{a_jM}$ is via
multiplication by $\mu_{\theta}^{a_jM}$. Since neither $a_j$
nor $M$ are divisible by $p$, we have $\mu_{\theta}^{a_jM}\neq
1$. Hence $\Gamma_{n,p}$ acts effectively on each of the line
bundles $L_{n,p}^{a_jM}$.
\end{proof}

Since $\rk G_j(\delta_j)=\rk F_j(\delta_j)=jn!$ 
we have
$$\rk V_{p,n}=n+1+\frac{n(n+1)}{2}n!.$$
Since the right hand side is independent of $p$, we may use Lemma \ref{lemma:trivial-bundles}
to conclude the existence of some natural number $\tau=\tau(n)$, depending on $n$ but not on $p$,
with the property that
$$V_{p,n}\oplus\uC^{\tau-\rk V_{p,n}}\simeq \uC^{\tau}$$
as vector bundles. Taking the trivial lift of the action of $\Gamma_{n,p}$ to
$\uC^{\tau-\rk V_{p,n}}$, we obtain an action of $\Gamma_{n,p}$ on $\uC^{\tau}$
which by the previous lemma is effective as soon as $p>M$.
Since $M$ only depends on $n$, the proof of Theorem \ref{thm:vector-bundles} is complete.

\section{Proof of Theorem \ref{thm:main}}

Set the functions $\tau,M:\NN\to\NN$ to be those of Theorem
\ref{thm:vector-bundles}. By Theorem \ref{thm:vector-bundles},
for any prime $p$ satisfying $p>M(n)$ and $p\equiv 1\mod n+1$
there is an effective action of $\Gamma_{n,p}$ on
$W:=T^{2n}\times\CC^{\tau(n)}$ by vector bundle automorphisms.
By Lemma \ref{lemma:accio-fibrat-trivial} the induced action on
the determinant line bundle
$T^{2n}\times\Lambda^{\tau(n)}\CC^{\tau(n)}$ admits an
equivariant nowhere vanishing section
$\sigma:T^{2n}\to\Lambda^{\tau(n)}\CC^{\tau(n)}$. Using the
standard averaging trick, we may take a
$\Gamma_{n,p}$-invariant Hermitian structure $h_0$ on $W$.
Multiplying $h_0$ by $|\sigma|_{h_0}^{-1/\tau(n)}\in\RR_{>0}$
we get an invariant Hermitian metric $h$ with respect to which
$\sigma$ has constant norm equal to $1$.
Then the bundle $E$ of $h$-unitary frames
$((\theta,w_1),\dots,(\theta,w_{\tau(n)}))\in W$ such that
$$w_1\wedge\dots\wedge w_{\tau(n)}=\sigma(\theta)$$ is
isomorphic to $T^{2n}\times\SU(\tau(n))$ and it carries an
effective action of $\Gamma_{n,p}$ by principal bundle
automorphisms. Hence there exists a cocycle
$$c:\Gamma_{n,p}\times T^{2n}\to\SU(\tau(n))$$
such that for any $\gamma\in\Gamma_{n,p}$ and any $(\theta,h)\in T^{2n}\times\SU(\tau(n))$
we have
$$\gamma\cdot(\theta,h)=(\gamma\cdot\theta,c(\gamma,\theta)h),$$
where $\Gamma_{n,p}$ acts on $T^{2n}$ via the map
$\eta:\Gamma_{n,p}\to\ZZ_p^{2n}$ in (\ref{eq:exact}) and
formula (\ref{eq:accio}). The cocycle condition is
$c(\gamma'\gamma,\theta)=c(\gamma',\gamma\theta)c(\gamma,\theta)$
for any $\theta\in T^{2n}$ and any
$\gamma,\gamma'\in\Gamma_{n,p}$. The fact that the action of
$\Gamma_{n,p}$ is effective is equivalent to the condition that
for any nontrivial $\gamma\in\Ker\eta$ and any $\theta\in
T^{2n}$ the element $c(\gamma,\theta)$ is nontrivial.

\subsection{Proof of (1)}

Suppose that $X$ is a manifold with an almost effective action of $\SU(\tau(n))$.
Let $H$ be the kernel of this action. Define an action of $\Gamma_{n,p}$ on
$T^{2n}\times X$ by setting
$$\gamma\cdot(\theta,x)=(\gamma\cdot\theta,c(\gamma,\theta)\cdot x)$$
for every $\gamma\in\Gamma_{n,p}$ and any $x\in X$.
The fact that this defines an action of $\Gamma_{n,p}$
follows from the cocycle condition satisfied by $c$.

\begin{lemma}
If $p$ does not divide $|H|$ then the action of $\Gamma_{n,p}$
on $T^{2n}\times X$ is effective.
\end{lemma}
\begin{proof}
If an element $\gamma\in\Gamma$ acts trivially on $T^{2n}\times X$ then we must have
$\eta(\gamma)=0$, so $c(\gamma,\theta)$ is a nontrivial element of order $p$ for every $\theta$
(indeed, if $\gamma\cdot\theta=\theta$ then the cocycle condition reads $c(\gamma^k,\theta)=c(\gamma,\theta)^k$ for every $k$).
If $p$ does not divide $|H|$ then $c(\theta,\gamma)$ does not belong to $H$, which implies that
$\gamma$ acts nontrivially on $T^{2n}\times X$.
\end{proof}

Setting $\Gamma:=\Gamma_{n,p}$, Lemma
\ref{lemma:abelian-subgroups-Gamma-n-p} concludes the proof of
statement (1) of the theorem.

\subsection{Proof of (2)}
Now assume that $r>1$. Then $E^r$ can be identified with a
trivial principal $\SU(\tau(n))^r$ bundle over $T^{2nr}$, and
it carries an effective action of $(\Gamma_{n,p})^r$.
We next prove that $(\Gamma_{n,p})^r$ contains a subgroup
$\Gamma$ with $p^{2n+r}$ elements which does not contain any
abelian subgroup with more than $p^{2+r+4n/r}$ elements. The
following result is due to Olshanskii (see Lemma 2 in
\cite{O}):

\begin{lemma}
Suppose that $k$ satisfies the condition
$4n<r(k-1)$. Then there exists a set of symplectic forms
$\{\omega_1,\dots,\omega_r\}$ in $V:=\ZZ_p^{2n}$ with the property that
no $k$-dimensional subspace of $V$ is isotropic for all the forms
$\omega_1,\dots,\omega_r$ simultaneously.
\end{lemma}

Let $\{\omega_1,\dots,\omega_r\}$ be as in the lemma, with $k=2+4n/r$.
Let $A_1,\dots,A_r$ be elements of $\GL(2n,\ZZ_p)$ such that
$$\omega_j(u,v)=\omega(A_ju,A_jv)$$
for each $j$ and each $u,v\in\ZZ_p^{2n}$, where $\omega$ is the standard symplectic
form (\ref{eq:omega}). Define
$$\Gamma=\{(\gamma_1,\dots,\gamma_r)\in (\Gamma_{n,p})^r\mid
A_1^{-1}\eta(\gamma_1)=A_2^{-1}\eta(\gamma_2)=\dots=A_r^{-1}\eta(\gamma_r)\},$$
where $\eta:\Gamma_{n,p}\to\ZZ_p^{2n}$ is the morphism in the exact sequence (\ref{eq:exact}).
Consider the projection
$$\eta':\Gamma\to\ZZ_p^{2n},\qquad \eta'(\gamma_1,\dots,\gamma_r)=A_1^{-1}\eta(\gamma_1).$$
Then there is an exact sequence
$$0\to\ZZ_p^r\to\Gamma\stackrel{\eta'}{\longrightarrow}\ZZ_p^{2n}\to 0,$$
and two elements $\gamma=(\gamma_1,\dots,\gamma_r)$ and $\gamma'=(\gamma'_1,\dots,\gamma'_r)$
of $\Gamma$ commute if, for each $j$,
\begin{multline*}
$$0=\omega(\eta(\gamma_j),\eta(\gamma_j'))=
\omega(A_jA_j^{-1}\eta(\gamma_j),A_jA_j^{-1}\eta(\gamma_j')) = \\ =
\omega(A_jA_1^{-1}\eta(\gamma_1),A_jA_1^{-1}\eta(\gamma_1'))=
\omega(A_j\eta'(\gamma),A_j\eta'(\gamma'))=\omega_j(\eta'(\gamma),\eta'(\gamma')).
\end{multline*}
So if $A\subset\Gamma$ is an abelian subgroup then
$\eta'(A)$ is a subspace of $\ZZ_p^{2n}$ which is isotropic with respect to all forms $\omega_1,\dots,\omega_r$ simultaneously. Hence, $\eta'(A)$ has dimension at most $k$
and consequently $A$ contains at most $p^{r+k}$ elements.

Now the proof of (2) is concluded arguing exactly as in (1),
replacing $\SU(\tau(n))$ by $\SU(\tau(n))^r$.


\begin{thebibliography}{ABCD}\frenchspacing\smallbreak

\bibitem{AH}
M.F. Atiyah, F. Hirzebruch,
Vector bundles and homogeneous spaces,
1961 {\em Proc. Sympos. Pure Math.}, Vol. III pp. 7–38
American Mathematical Society, Providence, R.I.

\bibitem{B} D. Bump, {\em Lie groups},
Graduate Texts in Mathematics, 225. Springer-Verlag, New York, 2004.

\bibitem{CPS} B. Csik\'os, L. Pyber, E. Szab\'o, Diffeomorphism groups of compact
$4$-manifolds are not always Jordan, preprint {\tt arXiv:1411.7524}.

\bibitem{DS} H. Donnelly, R. Schultz,
Compact group actions and maps into aspherical manifolds, {\em Topology}
{\bf 21} (1982), no. 4, 443--455.

\bibitem{GLO} D. Gottlieb, K.B. Lee, M. \"Ozaydin,
Compact group actions and maps into $K(\pi, 1)$-spaces,
{\em Trans. Amer. Math. Soc.} {\bf 287} (1985), no. 1, 419--429.

\bibitem{H}
D. Husemoller, {\em Fibre bundles. Third edition.}, Graduate Texts in Mathematics {\bf 20} Springer-Verlag, New York (1994).

\bibitem{MS}
L. N. Mann, J. C. Su, Actions of elementary p-groups on manifolds,
{\em Trans. Amer. Math. Soc.} {\bf 106} (1963), 115--126.


\bibitem{M1} I. Mundet i Riera, Jordan's theorem for the
    diffeomorphism group of some manifolds, Proc. AMS {\bf 138}
    (2010) 2253--2262.

\bibitem{M2} I. Mundet i Riera, Finite group actions on
    4-manifolds with nonzero Euler characteristic, preprint
    {\tt arXiv:1312.3149}.

\bibitem{M3} I. Mundet i Riera, Finite groups actions on
    manifolds without odd cohomology, preprint {\tt
    arXiv:1310.6565}.

\bibitem{M4} I. Mundet i Riera,
    Finite group actions on homology spheres and manifolds
    with nonzero Euler characteristic, preprint {\tt
    arXiv:1403.0383}.

\bibitem{MT} I. Mundet i Riera, A. Turull, Boosting an
    analogue of Jordan's theorem for finite groups, preprint
    {\tt  arXiv:1310.6518}.

\bibitem{NZM} I. Niven, H.S. Zuckerman, H.L. Montgomery,
{\em An introduction to the theory of numbers}, Fifth edition, John Wiley \& Sons, Inc., New York, 1991.

\bibitem{O}
A.Ju. Olshanskii,
On the question of the orders and the number of generators of abelian subgroups of finite p-groups, (Russian) {\em Mat. Zametki} {\bf 23} (1978), no. 3, 337--341. English translation: {\em Math. Notes} {\bf 23} (1978), no. 3--4, 183--185.

\bibitem{P}
E. Park,
{\em Complex topological K-theory},
Cambridge Studies in Advanced Mathematics {\bf 111},
Cambridge University Press, Cambridge, 2008.

\bibitem{Po0} V.L. Popov, On the Makar-Limanov, Derksen
    invariants, and finite automorphism groups of algebraic
    varieties. In {\em Peter Russell's Festschrift, Proceedings of
    the conference on Affine Algebraic Geometry held in
    Professor Russell's honour}, 1–5 June 2009, McGill Univ.,
    Montreal., volume 54 of Centre de Recherches Math\'ematiques
    CRM Proc. and Lect. Notes, pages 289–311, 2011.

\bibitem{Po1} V.L. Popov, Finite subgroups of diffeomorphism
    groups, {\em preprint }{\tt arXiv:1310.6548}.

\bibitem{WW} R. Washiyama, T. Watabe,
On the degree of symmetry of a certain manifold,
{\em J. Math. Soc. Japan} {\bf 35} (1983), no. 1, 53--58.

\bibitem{Za} Y.G. Zarhin, Theta groups and products of abelian and rational varieties,
{\em Proc. Edinb. Math. Soc.} {\bf 57} (2014) 299--304.

\bibitem{Z} B. Zimmermann, On Jordan
    type bounds for finite
    groups acting on compact $3$-manifolds,
    {\em Arch. Math.} {\bf 103} (2014), 195--200.
\end{thebibliography}
\end{document}